\documentclass[leqno]{amsart}
\usepackage{amsmath}
\usepackage{amssymb}
\usepackage{amsthm}
\usepackage{enumerate}
\usepackage[mathscr]{eucal}
\theoremstyle{plain}
\newtheorem{theorem}{Theorem}[section]
\newtheorem{prop}[theorem]{Proposition}

\newtheorem{lemma}{Lemma}[section]

\theoremstyle{definition}
\newtheorem{definition}{Definition}[section]
\newtheorem{remark}{Remark}[section]

\newtheorem{example}{Example}[theorem]

\usepackage[pagewise]{lineno}
\begin{document}
\title[ A complete characterization of Birkhoff-James orthogonality]{ A complete characterization of Birkhoff-James orthogonality in infinite dimensional normed space}
\author[Debmalya Sain, Kallol Paul and Arpita Mal]{Debmalya Sain, Kallol Paul and Arpita Mal}

\newcommand{\acr}{\newline\indent}

\address[Sain]{Department of Mathematics\\ Indian Institute of Science\\ Bengaluru 560012\\ Karnataka \\India\\ }
\email{saindebmalya@gmail.com}

\address[Paul]{Department of Mathematics\\ Jadavpur University\\ Kolkata 700032\\ West Bengal\\ INDIA}
\email{kalloldada@gmail.com}

\address[Mal]{Department of Mathematics\\ Jadavpur University\\ Kolkata 700032\\ West Bengal\\ INDIA}
\email{arpitamalju@gmail.com}

\thanks{Research of the first author is sponsored by Dr. D. S. Kothari Postdoctoral Fellowship, under the mentorship of Professor Gadadhar Misra. First author would like to lovingly acknowledge the blissful presence of his friends from Ramakrishna Mission Vidyapith Purulia, in general, and Mr.Arijeet Roy Chowdhury, in particular, in every sphere of his life!
The third author would like to thank UGC, Govt. of India for the financial support.} 


\begin{abstract}
In this paper, we study Birkhoff-James orthogonality of bounded linear operators and give a complete characterization of Birkhoff-James orthogonality of bounded linear operators on infinite dimensional real normed linear spaces.  As  an application of the results obtained, we prove a simple but useful characterization of Birkhoff-James orthogonality of bounded linear functionals defined on a real normed linear space, provided the dual space is strictly convex. We also provide  separate necessary and sufficient conditions for smoothness of bounded linear operators on infinite dimensional normed linear spaces.  
\end{abstract}

\maketitle
\section{Introduction.} 
Birkhoff-James orthogonality \cite{B, Ja} of elements in a normed linear space was introduced by Birkhoff in \cite{B}, in order to generalize the concept of orthogonality in inner product spaces. Over the years, Birkhoff-James orthogonality has been undoubtedly established as an important concept in the study of geometry of normed linear spaces by virtue of its rich connection with several geometric properties of the space, like strict convexity, smoothness etc. \cite{J, Ja, Jb}. Recently, a renewed interest has been generated towards studying the Birkhoff-James orthogonality of elements in the space of bounded linear operators between normed linear spaces \cite{BS, PSG, SP}. While complete characterization of Birkhoff-James orthogonality of bounded linear operators defined on a Hilbert space \cite{BS, P}, or a finite dimensional real Banach space \cite{S} has been obtained, the problem of characterizing Birkhoff-James orthogonality of bounded linear operators on infinite dimensional normed linear spaces remains unsolved. Our present paper settles the issue in that direction. Also a complete characterization of the smoothness of bounded linear operators on infinite dimensional normed linear spaces remains elusive for long despite having been studied by many mathematicians like \cite{A,DK,HR,HF,HO,KY,R} over the years. Recently in \cite{PSG}, a sufficient condition for the smoothness of a bounded linear operator has been obtained using Birkhoff-James orthogonality techniques. We carry on this work and prove separate  necessary and sufficient conditions for smoothness of bounded linear operators on infinite dimensional normed linear spaces for the first time.  Without further ado, let us establish our notations and terminologies.\\

In this paper, letters $ \mathbb{X}, \mathbb{Y} $ denote normed linear spaces. Throughout the paper, we will be working with real normed linear spaces. Let $B_{\mathbb{X}} = \{x \in \mathbb{X} \colon \|x\| \leq 1\}$ and
$S_{\mathbb{X}} = \{x \in \mathbb{X} \colon \|x\|=1\}$ be the unit ball
and the unit sphere of $\mathbb{X}$ respectively. Let $\mathbb{B}(\mathbb{X}, \mathbb{Y}) (\mathbb{K}(\mathbb{X}, \mathbb{Y})) $ denote the space of all bounded (compact) linear operators from $ \mathbb{X} $ to $ \mathbb{Y}. $ For $x,y \in \mathbb{X}$, $x$ is said to be
\emph{orthogonal to $y$ in the sense of Birkhoff-James} \cite{B}, written as $x \perp_B y$, if $\|x\| \leq \|x+\lambda y\|$ for all $\lambda \in \mathbb{R}$. Likewise for $T,A \in \mathbb{B}(\mathbb{X}, \mathbb{Y})$, $T$ is said to be orthogonal to $A$ in the sense of Birkhoff-James, written as $T \perp_B A$, if $\|T\| \leq \|T + \lambda A\|$ for all $\lambda \in  \mathbb{R}$. It is easy to observe that in inner product spaces, $x \perp_B y$ is equivalent to the usual inner product orthogonality $x \perp y$.

\smallskip

In a  Hilbert space $\mathbb{H}$, Bhatia and \v{S}emrl \cite{BS}  and Paul \cite{P} independently proved that $T \perp_B A$ if and only if there exists $\{x_n\} \subset S_{\mathbb{H}}$ such that $\|Tx_n\| \longrightarrow \|T\|$ and $\langle Tx_n, Ax_n \rangle \longrightarrow 0$.  If the space is finite dimensional it then follows that $T \bot_B A$ if and only if there exists $ x \in S_\mathbb{H}$ such that $ \|Tx\| = \|T\| $ and $  \langle Tx, Ax \rangle = 0.$  Recently in \cite{S}, Sain characterized the Birkhoff-James orthogonality of linear operators on a finite dimensional real Banach space . The following two definitions were necessary to obtain the desired characterization in \cite{S}: \\

\begin{definition}
Let $T \in \mathbb{B}(\mathbb{X}, \mathbb{Y})$.  We define $M_T$ to be the set of all unit vectors in $S_\mathbb{X}$
at which $T$ attains norm, i.e.,
\[
M_T = \{ x \in S_\mathbb{X} \colon \|Tx\| = \|T\| \}.
\]
\end{definition}
\begin{definition}
For any two elements $ x, y $ in a real normed linear space $ \mathbb{X}, $ let us say that $ y \in x^{+} $ if $ \| x + \lambda y \| \geq \| x \| $ for all $ \lambda \geq 0. $ Accordingly, we say that $ y \in x^{-} $ if $ \| x + \lambda y \| \geq \| x \| $ for all $ \lambda \leq 0$.
\end{definition}

In this paper, our aim is to extend the works of \cite{S} to the setting of infinite dimensional normed linear spaces. We extend Theorem $ 2.2 $ of \cite{S} completely, in case of compact linear operators defined on a reflexive Banach space. However,  the scenario is far more complicated in case of general bounded linear operators defined on a normed linear space. At this point of our discussion, the following definitions are in order. 

\begin{definition}
\cite{G} Let $ \mathbb{X} $ be a normed linear space and let $ x \in S_{\mathbb{X}}. $ We say that $ x $ is a rotund point of $ B_{\mathbb{X}} $ if $ \| y \| = \| \frac{x+y}{2} \| = 1 $ implies that $ x = y. $
\end{definition}

\begin{definition} A normed linear space $ \mathbb{X} $ is said to be strictly convex if for each $x,y \in S_{\mathbb{X}},$  $\|x+y\| < \|x\| + \|y\|$ whenever $x,y$ are linearly independent.
\end{definition}

\begin{definition} A normed linear space $ \mathbb{X} $ is said to be uniformly convex if to each $\epsilon, 0 < \epsilon \leq 2, $ there corresponds a $\delta(\epsilon) > 0$ such that the conditions $\|x\| = \|y\| = 1$ and $\|x - y \| > \epsilon $ imply $ \frac{\|x+y\|}{2} \leq  1-\delta(\epsilon).$ 
\end{definition}

\begin{definition}
\cite{PSJ} Let $ \mathbb{X} $ be a normed linear space and let $ x, y \in \mathbb{X}. $ We say that $ x $ is strongly orthogonal to $ y $ in the sense of Birkhoff-James, written as $ x \perp_{SB} y, $ if $ \| x + \lambda y \| > \| x \| $ for all $ \lambda \neq 0. $
\end{definition}

Following the  idea involved in the proof of Theorem 2.4 of \cite{SPJ}, we observe that $ x \in S_{\mathbb{X}} $ is a rotund point of $ B_{\mathbb{X}} $ if and only if for any $ y \in \mathbb{X} \setminus \{\theta\}, $ $ x \perp_{B} y $ implies that $ x \perp_{SB} y. $ Equipped with this characterization of rotund points, we proceed towards obtaining a complete characterization of Birkhoff-James orthogonality of rotund points in the space of bounded linear operators. In order to obtain the desired characterization for rotund points and for general bounded linear operators, we need to introduce a new definition which is essentially geometric in nature. First, let us give a brief motivation in this regard. The notion of approximate orthogonality ($ \epsilon $-orthogonality) was first considered by Chmieli\'nski in \cite{C}: \\
Let $\mathbb{H}$ be an inner product space and let $ x, y \in \mathbb{H}. $ For $ \epsilon \in [0, 1), $ we say that $ x $ is $ \epsilon $-orthogonal to $y$, written as $ x \perp^{\epsilon} y, $ if $ |<x, y>| \leq \epsilon \| x \| \| y \|. $ The definition was suitably modified in \cite{Ca}, to obtain an analogous definition of $ \epsilon $-orthogonality in normed linear spaces: 

\begin{definition}
\cite{Ca} Let $ \mathbb{X} $ be a normed linear space and let $ x, y \in \mathbb{X}. $ For $ \epsilon \in [0, 1)$, we say that $ x $ is $ \epsilon $-orthogonal to $ y $ (in the sense of Birkhoff-James), written as $ x \perp_{D}^{\epsilon} y, $ if $ \| x+\lambda y \|\geq \sqrt[]{1-\epsilon^2}~\|x\|$ for all $\lambda \in \mathbb{R}. $
\end{definition}

In \cite{S} we ``decomposed" Birkhoff-James orthogonality (via Definition $ 1.2 $ stated in this paper) in order to obtain a complete characterization of Birkhoff-James orthogonality of bounded linear operators on finite dimensional Banach space. Following similar motivations, in this paper we decompose $ \epsilon $-orthogonality in order to completely characterize Birkhoff-James orthogonality of bounded linear operators, by means of the following definition.

\begin{definition}
Let $ \mathbb{X} $ be a normed linear space and let $ x, y \in \mathbb{X}. $ For $ \epsilon \in [0, 1), $ we say that $ y \in x^{+(\epsilon)} $if $ \| x+\lambda y \|\geq \sqrt[]{1-\epsilon^2}~\|x\|$ for all $\lambda \geq 0. $ Similarly, we say that $ y \in x^{-(\epsilon)} $if $ \| x+\lambda y \|\geq \sqrt[]{1-\epsilon^2}~\|x\|$ for all $\lambda \leq 0. $ 
\end{definition}

Motivated by the result on rotund bounded linear operators, we finally obtain a  complete characterization of Birkhoff- James orthogonality of bounded linear operators on general normed linear spaces. As an application of the results obtained by us, we completely characterize Birkhoff-James orthogonality of bounded linear functionals on a normed linear space whose dual is strictly convex. 

\medskip

The study of the geometry of the space of bounded linear operators on a general normed linear space $\mathbb{X}$ is far more complicated than that of the ground space $\mathbb{X}.$ The relation of Birkhoff-James orthogonality of bounded linear operators   with that of some special elements of $\mathbb{X}$ has been used in \cite{PSG} to obtain a sufficient condition for smoothness of a bounded linear operator. Carrying on in this direction we here obtain separate necessary and sufficient conditions for smoothness of a bounded linear operator on a general normed linear space. An element $ x \in S_{\mathbb{X}} $ is said to be a smooth point if there exists a unique supporting hyperplane to $B_{\mathbb{X}}$ at $x.$ The characterization of smooth points obtained by James \cite{J} 
has been used in our study, which states that $ x \in S_{\mathbb{X}} $ is a smooth point if and only if $ x \bot_B y $ and $ x \bot_B z$ implies $ x \bot_B {(y + z)}.$

\section{Main results.}
Let us begin by giving a complete characterization of Birkhoff-James orthogonality of compact linear operators defined on a reflexive Banach space. This extends Theorem $ 2.2 $ of \cite{S}.
\begin{theorem}\label{theorem:compact}
Let $\mathbb{X}$ be a reflexive Banach space and  $\mathbb{Y}$ be any normed linear space. Then for any $T, A\in \mathbb{K}(\mathbb{X}, \mathbb{Y})$, $T\bot_B A$ if and only if there exists $x, y\in M_T$ such that $Ax\in (Tx)^+$ and $Ay\in (Ty)^-$.
\end{theorem}
\begin{proof}
Let us first prove the easier sufficient part. \\
Since $Ax\in (Tx)^+,~\|T+\lambda A\|\geq \|Tx+ \lambda Ax\|\geq \|Tx\|=\|T\|$ for all $\lambda \geq 0$. Similarly $Ay\in (Ty)^-$ implies that $\|T+\lambda A\|\geq \|Ty+ \lambda Ay\|\geq \|Ty\|=\|T\|$ for all $\lambda \leq 0$. This completes the proof of the sufficient part. \\
Let us now prove the necessary part. \\
Since $T$ and $A$ are compact linear operators, $(T+\frac{1}{n} A)$ is also a compact linear operator for each $ n \in \mathbb{N}. $ Since $\mathbb{X}$ is reflexive, $(T+\frac{1}{n} A)$ attains norm for each $n \in \mathbb{N}$. \\
Therefore, for each $n\in \mathbb{N},$ there exists $x_n\in S_\mathbb{X}$ such that $\|(T+ \frac{1}{n}A)x_n\|=\|T+ \frac{1}{n}A\|$. \\
Since $\mathbb{X}$ is reflexive, $B_\mathbb{X}$ is weakly compact. Therefore $\{x_n\}$ has a weakly convergent subsequence. Without loss of generality we may assume that $\{x_n\}$ weakly converges to $x$ (say). Since $T$ and $A$ are compact linear operators, $Tx_n\rightarrow Tx$ and  $Ax_n\rightarrow Ax$. Since $T\bot_B A$, $\|T+ \frac{1}{n}A\|\geq \|T\|$ for all $n\in \mathbb{N}$. \\
Therefore, $\|(T+ \frac{1}{n}A)x_n\|=\|T+ \frac{1}{n}A\|\geq \|T\|\geq \|Tx_n\|$ for all $n\in \mathbb{N}$. Letting $n\rightarrow \infty$, we see that, $\|Tx\|\geq \|T\|\geq \|Tx\|$. This proves that $\|Tx\|=\|T\|,$ i.e., $x\in M_T$. \\
Now we show that $Ax\in (Tx)^+$. \\
For any $\lambda \geq \frac{1}{n},$ we claim that $\|Tx_n+ \lambda Ax_n\|\geq \|Tx_n\|$. \\
Otherwise, $Tx_n+ \frac{1}{n} Ax_n= (1- \frac{1}{n\lambda}) Tx_n+ \frac{1}{n\lambda}(Tx_n+ \lambda Ax_n)$ gives that, \\
 $ \|T+\frac{1}{n}A\|=\|Tx_n+ \frac{1}{n} Ax_n\|\leq (1- \frac{1}{n\lambda}) \|Tx_n\|+ \frac{1}{n\lambda}\|(Tx_n+ \lambda Ax_n)\|<(1- \frac{1}{n\lambda}) \|Tx_n\|+ \frac{1}{n\lambda}\|Tx_n\|= \|Tx_n\| \leq \|T\|$, a contradiction. This completes the proof of our claim. \\
Now for any $\lambda> 0$, there exists $n_0 \in \mathbb{N}$ such that $\lambda>\frac{1}{n_0}$. So for all $n\geq n_0$, $\|Tx_n+ \lambda Ax_n\|\geq \|Tx_n\|$. Therefore, letting $n\rightarrow \infty$, we have, $\|Tx+ \lambda Ax\|\geq \|Tx\|$. This completes the proof of the fact that $Ax\in (Tx)^+$. \\
Similarly, considering the compact operators $T-\frac{1}{n} A,$ it is now easy to see that there exists $y\in M_T$ such that $Ay\in (Ty)^-$. This completes the proof.
\end{proof}
\begin{remark}
Theorem 2.2 of \cite{S} now follows easily as a simple consequence of the above theorem since every finite dimensional normed linear space is reflexive and every bounded linear operator there is a compact linear operator.
\end{remark}
For bounded linear operators defined on a normed linear space, the situation is far more complicated since in this case the norm attainment set may be empty. In the next proposition, we give a  sufficient condition for Birkhoff-James orthogonality of  bounded linear operators. 
\begin{prop}
Let $\mathbb{X}$ and $\mathbb{Y}$ be normed linear spaces. Let $T$, $A\in \mathbb{B}(\mathbb{X}, \mathbb{Y})$. Suppose there exists two sequences $\{x_n\}$ and $\{y_n\}$ in $S_\mathbb{X}$ satisfying the following two conditions: \\
$(i)\| Tx_n\| \rightarrow \|T\|$ and $\|Ty_n\|\rightarrow\|T\|$, as $n \rightarrow \infty $ \\
$(ii)Ax_n\in(Tx_n)^+$ and $Ay_n\in(Ty_n)^-$ for all  $n\in \mathbb{N}$. \\
Then $T\bot_B A$. 
\end{prop}
\begin{proof}
Since $Ax_n\in (Tx_n)^+$, for any $\lambda \geq 0$ we have, $\|T+ \lambda A\| \geq \|Tx_n+ \lambda Ax_n\| \geq \|Tx_n\|$ for all $n \in \mathbb{N}$. Therefore letting $n\rightarrow \infty$, we have, $\|T+ \lambda A\| \geq \|T\|$, since $\| Tx_n\| \rightarrow \|T\|$ as $n\rightarrow \infty$. \\
Similarly, $Ay_n\in (Ty_n)^{-}$ implies that, for any $\lambda \leq 0$, $\|T+ \lambda A\| \geq \|Ty_n+ \lambda Ay_n\| \geq \|Ty_n\|$ for all $n \in \mathbb{N}$. Therefore letting $n\rightarrow \infty$, we have, $\|T+ \lambda A\| \geq \|T\|$, since $\| Ty_n\| \rightarrow \|T\|$ as $n\rightarrow \infty$. This completes the proof of the fact that  $T\bot_B A$.  
\end{proof}
In the next example we illustrate the fact that the conditions stated in Proposition $ 2.2 $ are only sufficient but not necessary for $ T \perp_B A. $
\begin{example}
Define $T, A : l_1 \rightarrow l_1$ by 
$Te_n = (1- \frac{1}{n+1})e_n$,  $n\geq 1$, and 
$Ae_n = \frac{1}{n+1} e_n$, $n\geq 1$, where $e_n =(0,0, \ldots ,0,1,0, \ldots ),$ with $1$ in the $n$-th position and $0$ elsewhere. \\
Let $x= \sum \limits_{n=1}\limits ^\infty a_n e_n \in l_1$ where $a_n \in \mathbb{R}$ for all $n$. Then $\|Tx\| =\sum \limits_{n=1}\limits ^\infty |a_n| |1- \frac{1}{n+1}| \leq \sum \limits_{n=1}\limits ^\infty |a_n| = \|x\|$. So $\|T\| \leq 1$. Also $\|Te_n\|= (1- \frac{1}{n+1}) \rightarrow 1$ as $n \rightarrow \infty$. Hence $\|T\|= 1$. \\
First we show that $ T \perp_B A.$ For any scalar $ \lambda  $, $ \| T + \lambda A \| \geq \| ( T + \lambda A ) e_n \| = | 1 - \frac{(1 - \lambda)}{n+1} | \longrightarrow 1 $, so that $ \| T + \lambda A \| \geq \|T\| $ for all $ \lambda $. 
Now for  $ x \in l_1 \setminus \{0\} , $ we get $\|(T- A) x\|= \sum \limits_{n=1}\limits ^\infty |a_n| |1- \frac{2}{n+1} | <  \sum \limits_{n=1}\limits ^\infty |a_n| |1- \frac{1}{n+1}| = \|Tx\|$, which implies that $\|(T- A) x\| < \|Tx\|$ for all $x\in l_1 \setminus \{0\}$. In particular, we must have, $Ax \notin (Tx)^-$ for all $x\in l_1 \setminus \{0\}$. \\
Thus, we observe that in this example, although $ T \perp_{B} A, $ it is not possible to find a sequence $ \{y_{n}\} $ in $ S_{\mathbb{X}} $ such that $ \| Ty_n \| \rightarrow \|T\| $ as $ n \rightarrow \infty $ and $Ay_n\in(Ty_n)^-$ for all  $n\in \mathbb{N}$.
\end{example}

Next we characterize Birkhoff-James orthogonality of rotund points in the space of bounded linear operators. First we need the following lemma, which also gives a characterization of rotund points. Note that the proof of the lemma can be obtained similarly as the proof of Theorem $2.4$ of \cite{SPJ}.
\begin{lemma}\label{lemma:rotund}
Let $\mathbb{X}$ be a normed linear space.Then $x\in S_\mathbb{X}$ is a rotund point if and only if $x\bot_B y \Rightarrow x\bot_{SB} y$ for any $y\in \mathbb{X} \setminus \{\theta\}$. 
\end{lemma}

In the following theorem, we obtain the  characterization of Birkhoff-James orthogonality of rotund points in the space of bounded linear operators defined between infinite dimensional normed linear spaces.
\begin{theorem}\label{theorem:rotund}
Let $\mathbb{X}$ and $\mathbb{Y}$ be two normed linear spaces. Let $T$ be a rotund point of $\mathbb{B} (\mathbb{X}, \mathbb{Y})$. Then for any $A \in \mathbb{B} (\mathbb{X}, \mathbb{Y}) , ~T \bot_B A$ if and only if there exists two sequences $\{x_n\}$, $\{y_n\}$ in $S_\mathbb{X}$ and two sequences of positive real numbers $\{\epsilon_n\}$ , $\{\delta_n\}$ such that \\ 
$(i)  \epsilon_n \rightarrow 0$ , $\delta_n\rightarrow 0$ as $n\rightarrow \infty$ .\\
$(ii) \| Tx_n\| \rightarrow \|T\|$ and $\|Ty_n\|\rightarrow\|T\|$ as $n\rightarrow \infty$ .\\
$(iii)Ax_n\in(Tx_n)^{+(\epsilon_n)}$ and $Ay_n\in(Ty_n)^{-(\delta_n)}$ for all  $n\in \mathbb{N}$.
\end{theorem}
\begin{proof}
Let us first prove the easier sufficient part.\\
Since $ Ax_n\in(Tx_n)^{+(\epsilon_n)} $ for all $n\in \mathbb{N}, $ $\|Tx_n+\lambda Ax_n\|\geq \sqrt[]{1-\epsilon_n^2}~\|Tx_n\|$ for all $\lambda\geq0.$\\
This implies, given any $ \lambda \geq 0, $ $ \|T+\lambda A\|\geq\|(T+\lambda A)x_n\| \geq\sqrt[]{1-\epsilon_n^2} \|Tx_n\|$. Since $ \epsilon_n \rightarrow 0 $ and $ \|Tx_n\| \rightarrow \| T \| $ as $ n \rightarrow \infty, $ we obtain, 
     \[ \|T+\lambda A \|\geq \|T\|. \]
Similarly, $ Ay_n\in(Ty_n)^{-(\delta_n)} $ for all $n\in \mathbb{N}$ implies that, for any $ \lambda < 0, $ 
\[ \|T+\lambda A \|\geq \|T\|. \]
This completes the proof of the sufficient part. \\
Let us now prove the more involved necessary part. \\
Since $T\in \mathbb{B} (\mathbb{X}, \mathbb{Y})$ is a rotund point, clearly $ T $ is nonzero. Also by Lemma \ref{lemma:rotund}, $T\bot _B A\Rightarrow T\bot_{SB} A$ for any nonzero $ A \in \mathbb{B} (\mathbb{X}, \mathbb{Y}). $ \\
Therefore for each $n\in \mathbb{N}$, 
\[\|T+ \frac{1}{n}A\| >\|T\|. \]
This implies that for each $ n \in \mathbb{N}, $ there exists $x_n\in S_\mathbb{X}$ such that 
$\|(T+\frac{1}{n}A)x_n\|>\|T\|.$
We claim that $\|Tx_n\|\rightarrow \|T\|$. Indeed, $\|Tx_n\|=\|(T+\frac{1}{n}A)x_n - \frac{1}{n}Ax_n\|$ $\geq \|(T+\frac{1}{n}A)x_n\| - \frac{1}{n} \|Ax_n\|$ $>\|T\|- \frac{1}{n} \|A\|$ $\rightarrow \|T\|$ as $n \rightarrow \infty. $ Since $ x_n \in S_{\mathbb{X}},$ $ \| Tx_n \| \leq \| T \|  $. This proves our claim. \\
Since $\|Tx_n\| \rightarrow \|T\|>0,$ there exists $n_1\in \mathbb{N}$ such that $\|Tx_n\| > \frac{\|T\|}{2} > 0$ for all $n \geq n_1$. Choose $n_2 \in \mathbb{N}$ such that $n_2 > \frac{2 \|A\|}{\|T\|}$. Let $n_0 =$ max$\{ n_1,n_2\}$.
Then for all $n\geq n_0$, $0< \frac{\|A\|}{n \|Tx_n\|} < \frac{2 \|A\|}{n \|T\|}< 1$, which implies that for all $n\geq n_0$, $0< 1 - \frac{ \|A\|}{n \|Tx_n\|}< 1$. \\
Choose $\epsilon_n=\sqrt[]{1- (1- \frac{\|A\|}{n \|Tx_n\|})^2}$. Then clearly $\epsilon_n \rightarrow 0$ as $n \rightarrow \infty$ .\\
We claim that $Ax_n\in(Tx_n)^{+(\epsilon_n)}$ for all $n\geq n_0$ .\\
Let $n\geq n_0$. Then for $0< \lambda \leq \frac{1}{n}$, \\
$\|Tx_n+\lambda Ax_n\|\geq \|Tx_n\|-|\lambda | \|Ax_n\|\geq \|Tx_n\|-\frac{1}{n} \|A\|=\sqrt{1-\epsilon_n^2} \|Tx_n\|$. \\
For $\lambda > \frac{1}{n} > 0$, we claim that $\|Tx_n+ \lambda Ax_n\| \geq \|Tx_n\|. $ \\
Suppose on the contrary, we have, $\|Tx_n+ \lambda Ax_n\| < \|Tx_n\| $ for some $ \lambda > \frac{1}{n}. $ Now, there exists $ t \in (0, 1) $ such that $Tx_n+  \frac{1}{n} Ax_n = t(Tx_n)+ (1-t)(Tx_n+\lambda Ax_n)$. This implies that, $\|Tx_n+ \frac{1}{n} Ax_n \| \leq t \|Tx_n\|+ (1-t)\|Tx_n+ \lambda Ax_n\|< t \|Tx_n\| +(1-t) \|Tx_n\|= \|Tx_n\| \leq \|T\|$, a contradiction. \\
Therefore, for all $\lambda \geq 0$, $\|Tx_n+ \lambda Ax_n\| \geq \|Tx_n\|-\frac{1}{n} \|A\|=\sqrt[]{1-\epsilon_n^2} ~\|Tx_n\|$. This completes the proof of our claim. \\
Similarly, considering $\|T-\frac{1}{n} A\|> \|T\|$ for each $n \in \mathbb{N}$, we can find the desired sequences $\{y_n\} $ in $ S_{\mathbb{X}}$ and $\{\delta _n\} $ in $\mathbb{R}^+$ such that all the conditions $ (i), (ii) $ and $ (iii) $ are satisfied. 
\end{proof} 
Next we illustrate with an example to show that the conditions mentioned in Theorem \ref{theorem:rotund} is not sufficient to characterize Birkhoff-James orthogonality of bounded linear operators.
\begin{example}\label{example:not rotund}
Define $T, A : \ell_1 \rightarrow \ell_1$ by 
$ Te_1 = \frac{1}{2}e_1,~Te_n = (1- \frac{1}{n^4})e_n$,  $n\geq 2$, and 
$ Ae_1 = \frac{1}{2}e_1,~Ae_n = \frac{1}{n^2} (\frac{1}{n^2}-1)e_n$, $n\geq 2$, where $e_n =(0,0, \ldots ,0,1,0, \ldots ),$ with $1$ in the $n$-th position and $0$ elsewhere. 
Then both $T$ and $A$ are bounded linear operators on  $\ell_1$ with $\|T\| =1.$ Also $ T \bot_B A$, since $\|T + \lambda A \| \geq 
\| (T + \lambda A)e_n \| = \| (1 - \frac{1}{n^4})e_n + \lambda \frac{1}{n^2}(\frac{1}{n^2}-1) e_n \| = 
\mid 1 - \frac{1}{n^4} + \frac{\lambda}{n^4} - \frac{\lambda}{n^2} \mid \longrightarrow 1 = \|T\|.$ \\
Now for each $x = (a_1,a_2,\ldots,\dots) \in S_{\ell_1} $ we have
\begin{eqnarray*}
\|Tx + Ax\| & = & \| (a_1, (1-\frac{1}{2^2})a_2,\ldots,(1-\frac{1}{n^2})a_n,\ldots ) \| \\
            & = & \mid a_1 \mid + \mid a_2 \mid (1-\frac{1}{2^2}) + \ldots + \mid a_n \mid (1-\frac{1}{n^2}) +                  \ldots \\
						& = & \mid \frac{a_1}{2} \mid + \mid a_2 \mid (1-\frac{1}{2^4}) + \ldots + \mid a_n \mid (1-\frac{1}{n^4}) +                  \ldots \\
						&   & + \mid \frac{a_1}{2} \mid + \mid a_2 \mid (\frac{1}{2^4}-\frac{1}{2^2}) + \ldots + \mid a_n \mid                    (\frac{1}{n^4}-\frac{1}{n^2}) +                  \ldots \\
						& \leq & \|Tx\| + \frac{1}{2} + \sum_{n=2}^{\infty}(\frac{1}{n^4}-\frac{1}{n^2})\\
						& < & \|Tx\| - 0.06\\
						& \leq & \|T\|-0.06.						
\end{eqnarray*}
This shows that there does not exist $\{\epsilon_n\} $ with $\epsilon_n \longrightarrow 0 $ and $\{x_n\} \subset S_{\ell_1}$ such that $\|Tx_n\| \longrightarrow \|T\| $ and $Ax_n \in (Tx_n)^{+(\epsilon_n)}, $ for otherwise
\[ \sqrt{1 -{\epsilon_n }^2}||Tx_n\| \leq \|Tx_n + Ax_n\| < \|Tx_n\|-0.06, \]
and so letting $n \longrightarrow \infty$ we get a contradiction. Thus $T \bot_B A$ but $T$ does not satisfy hypothesis of Theorem \ref{theorem:rotund}.
\end{example}
The last example indicates that some additional condition is required to obtain
 the most generalized characterization of Birkhoff-James orthogonality of  bounded linear operators defined between infinite dimensional normed linear spaces. We accomplish the goal in the next theorem.
\begin{theorem}\label{theorem:bounded}
Let $\mathbb{X}$ and $\mathbb{Y}$ be two normed linear spaces. Let $T \in \mathbb{B} (\mathbb{X}, \mathbb{Y})$ be nonzero. Then for  any $A \in \mathbb{B} (\mathbb{X}, \mathbb{Y}) ,~ T \bot_B A$ if and only if either of the conditions in $(a)$ or in $(b)$ holds: \\
$(a)$ There exists a sequence $\{x_n\}$ in $S_\mathbb{X}$ such that $\| Tx_n\| \rightarrow \|T\|$ and $\|Ax_n\| \rightarrow 0$ as $n\rightarrow \infty$. \\
$(b)$ There exists two sequences $\{x_n\},~\{y_n\}$ in $S_\mathbb{X}$ and two sequences of positive real numbers $\{\epsilon_n\}$ , $\{\delta_n\}$ such that \\ 
$(i)  \epsilon_n \rightarrow 0$ , $\delta_n\rightarrow 0$ as $n\rightarrow \infty$ .\\
$(ii) \| Tx_n\| \rightarrow \|T\|$ and $\|Ty_n\|\rightarrow\|T\|$ as $n\rightarrow \infty$ .\\
$(iii)Ax_n\in(Tx_n)^{+(\epsilon_n)}$ and $Ay_n\in(Ty_n)^{-(\delta_n)}$ for all  $n\in \mathbb{N}$.
\end{theorem}
\begin{proof}
We first prove the easier sufficient part.\\
Suppose $(a)$ holds. Now for any scalar $\lambda$, $\|T + \lambda A\| \geq \|Tx_n + \lambda Ax_n\|\geq \|Tx_n\|- |\lambda | \|Ax_n\| \rightarrow \|T\|$ as $n \rightarrow \infty$. Therefore $T \bot_{B} A$. \\
Now suppose $(b)$ holds. Then following the same line of arguments as in the proof of the sufficient part of Theorem \ref{theorem:rotund}, we obtain $T\perp_B A$. 
This completes the proof of the sufficient part. \\
Let us now prove the comparatively trickier necessary part. \\
Suppose $(a)$ does not hold. \\
Without loss of generality let us assume that $\|A\| \leq 1$. \\
Since $T \bot_{B} A$, for any nonzero scalar $\lambda$, $\|T + \lambda A\| \geq \|T\|$. In particular, for each $n \in \mathbb{N}$, 
      \[\|T + \frac{1}{n} A\| > \|T\| - \frac{1}{n^3}. \]
Therefore, for each $n \in \mathbb{N}$, there exists a sequence $\{x_n\} $ in $ S_{\mathbb{X}}$ such that $\|(T+ \frac{1}{n} A) x_n \| > \|T\|- \frac{1}{n^3} \geq \|Tx_n\|- \frac{1}{n^3}$. \\
We claim that $\|Tx_n\|\rightarrow \|T\|$. Indeed, $\|Tx_n\|=\|(T+\frac{1}{n}A)x_n - \frac{1}{n}Ax_n\|$ $\geq \|(T+\frac{1}{n}A)x_n\| - \frac{1}{n} \|Ax_n\|$ $>\|T\|- \frac{1}{n^3} - \frac{1}{n} \|A\|$ $\rightarrow \|T\|$ as $n \rightarrow \infty. $ Since $ x_n \in S_{\mathbb{X}},$ $ \| Tx_n \| \leq \| T \|  $. This proves our claim. \\
Since $(a)$ does not hold, we assume that, $\inf \limits_{n \in \mathbb{N}}\|Ax_n\|= c > 0$. \\
Choose $n_1 \in \mathbb{N}$ such that $n_1 > \frac{2 \|T\|}{c}$. Since $\|Tx_n\| \rightarrow \|T\|>0,$ there exists $n_2\in \mathbb{N}$ such that $\|Tx_n\| > \frac{\|T\|}{2} > 0$ for all $n \geq n_2$. Choose $n_3 \in \mathbb{N}$ such that $n_3 > \frac{2}{\|T\|}$. Let $n_0 =$ max$\{ n_1,n_2, n_3\}$. Then for all $n\geq n_0$, $0< \frac{1}{n \|Tx_n\|} < \frac{2}{n \|T\|}< 1$, which implies that for all $n\geq n_0$, $0< 1 - \frac{1}{n \|Tx_n\|}< 1$. \\
Choose $\epsilon_n=\sqrt[]{1- (1- \frac{1}{n \|Tx_n\|})^2}$. Then clearly $\epsilon_n \rightarrow 0$ as $n \rightarrow \infty$ .\\
We claim that $Ax_n\in(Tx_n)^{+(\epsilon_n)}$ for all $n\geq n_0$ .\\
Let $n \geq n_0$. Then for $0 \leq \lambda < \frac{1}{n}$, \\
$\|Tx_n + \lambda Ax_n\| \geq \|Tx_n\|- \lambda \|Ax_n\| \geq \|Tx_n\|-\frac{1}{n}$. \\
For $\frac{1}{n} \leq \lambda \leq n$, we claim that $\|Tx_n +\lambda Ax_n\| \geq \|Tx_n\|- \frac{1}{n}$. \\
Suppose on the contrary, we have, $\|Tx_n +\lambda Ax_n\| < \|Tx_n\|- \frac{1}{n}$ for some $\frac{1}{n} \leq \lambda \leq n$. Now, $Tx_n + \frac{1}{n} Ax_n = (1- \frac{1}{n \lambda})Tx_n + \frac{1}{n \lambda}(Tx_n +\lambda Ax_n)$. This implies that, $\|Tx_n\|- \frac{1}{n^3} < \|Tx_n + \frac{1}{n} Ax_n\| \leq  (1- \frac{1}{n \lambda})\|Tx_n\| + \frac{1}{n \lambda}\|(Tx_n +\lambda Ax_n)\| < (1- \frac{1}{n \lambda})\|Tx_n\| + \frac{1}{n \lambda}(\|Tx_n\| - \frac{1}{n})= \|Tx_n\|- \frac{1}{n^2 \lambda}$. This implies that $\lambda > n$, a contradiction. \\
Thus for $0 \leq \lambda \leq n$, $\|Tx_n +\lambda Ax_n\| \geq \|Tx_n\|- \frac{1}{n}$. Hence, $0 \leq \lambda \leq \frac{2 \|T\|}{c}$ gives that $\|Tx_n +\lambda Ax_n\| \geq \|Tx_n\|- \frac{1}{n}$. \\
Now, for $\lambda > \frac{2 \|T\|}{c}$, $\|Tx_n +\lambda Ax_n\| \geq \lambda \|Ax_n\|- \|Tx_n\| \geq \lambda c- \|Tx_n\| > 2\|T\|-\|Tx_n\| \geq \|Tx_n\|- \frac{1}{n}$. \\
Therefore for all $\lambda \geq 0$, $\|Tx_n +\lambda Ax_n\| \geq \|Tx_n\|- \frac{1}{n}= \sqrt[]{1-\epsilon_n^2} ~\|Tx_n\|$. This completes the proof of our claim. \\
Similarly, considering $\|T-\frac{1}{n} A\|> \|T\|- \frac{1}{n^3}$ for each $n \in \mathbb{N}$, we can find the desired sequences $\{y_n\} $ in $ S_{\mathbb{X}}$ and $ \{\delta _n\} $ in $ \mathbb{R}^+$ such that all the conditions of $(b)$ are satisfied. 
\end{proof}

As an application of Theorem ~\ref{theorem:rotund} we give a complete characterization of Birkhoff-James orthogonality of bounded linear functionals on a normed linear space, whose dual is strictly convex. First, we need the following easy proposition. 

\begin{prop}\label{prop:real}
For any two real numbers $x$, $y$ and for any $\epsilon \in [0,1) $, $y\in x^{+(\epsilon)}$ if and only if $x y\geq 0$. Similarly, $y\in x^{- (\epsilon)}$ if and only if $x y\leq 0$. 
\end{prop} 
Now, the promised characterization:

\begin{theorem}\label{theorem:functional}
Let $\mathbb{X}$ be a normed linear space such that $\mathbb{X}^*$ is strictly convex. Then for any $f, g\in \mathbb{X}^*$, $f\bot_{B} g $ if and only if there exist $\{x_n\}$, $\{y_n\} $ in $ S_\mathbb{X}$ such that \\
$(i)|f(x_n)|\rightarrow \|f\|$ and $ |f(y_n)| \rightarrow \|f\|$ as $n \rightarrow \infty$ \\
$(ii)f(x_n).g(x_n) \geq 0$ and  $f(y_n).g(y_n) \leq 0$ for all $n\in \mathbb{N}$.
\end{theorem}
\begin{proof}
Since $ \mathbb{X}^{*} $ is strictly convex, $ f, g $ are rotund points of $ \mathbb{X}^{*}. $ The rest of the proof follows directly from Theorem ~\ref{theorem:rotund} and Proposition ~\ref{prop:real}. 
\end{proof}

Similarly as an application of Theorem \ref{theorem:bounded}, Birkhoff-James orthogonality of bounded linear functionals  on an infinite dimensional  normed linear space can be stated as follows.

\begin{theorem}\label{theorem:bounded functional}
Let $\mathbb{X}$ be a normed linear space. Then for any $f, g\in \mathbb{X}^*$, $f\bot_{B} g $ if and only if either of the conditions in (a) or in (b) holds:\\
(a) there exists $\{x_n\}$ in $ S_\mathbb{X}$ such that $\mid f(x_n)\mid \longrightarrow \|f\| $ and  $g(x_n) \longrightarrow 0.$\\
(b) there exists $\{x_n\}, \{y_n\}$ in $ S_\mathbb{X}$ such that \\
$(i)|f(x_n)| \rightarrow \|f\|$ and $ |f(y_n)| \rightarrow \|f\|$ as $n \rightarrow \infty$ \\
$(ii)f(x_n).g(x_n) \geq 0$ and  $f(y_n).g(y_n) \leq 0$ for all $n\in \mathbb{N}$.
\end{theorem}

\begin{remark}
If $ \mathbb{X} $ is reflexive then every bounded linear functional on $ \mathbb{X}, $ by virtue of being compact, attains norm. In that case, considering $ \mathbb{Y} = \mathbb{R} $ in Theorem ~\ref{theorem:compact} we get a complete description of the Birkhoff-James orthogonality of bounded linear functionals on $ \mathbb{X}. $ In fact, the following reformulation of Theorem ~\ref{theorem:compact} is worth mentioning in this context:

\begin{theorem}
Let $\mathbb{X}$ be a reflexive Banach space. Then for any $f, g \in \mathbb{X}^*$, $f\bot_B g$ if and only if there exists $ x, y \in M_f $ such that $ f(x).g(x) \geq 0 $ and $ f(y).g(y) \leq 0. $ In addition, if $ \mathbb{X} $ is strictly convex then  $f\bot_B g $ if and only if there exists $ x \in M_f $ such that $ f(x).g(x) = 0. $
\end{theorem}

\end{remark}

\begin{remark}
However, if $ \mathbb{X} $  is not reflexive then the norm attaining set of a bounded linear functional may be possibly empty and Theorem ~\ref{theorem:compact} is no longer applicable. In fact, it is well known that if $ \mathbb{X} $ is not reflexive then there exists a bounded linear functional on $ \mathbb{X} $ such that the functional does not attain norm. It is particularly in these cases that Theorem ~\ref{theorem:functional} can be effectively applied, provided the dual space $ \mathbb{X}^{*} $ is strictly convex. We would also like to note that there exists a non-reflexive Banach space whose dual is strictly convex. Indeed, the classical $ c_0 $ space (with the \emph{sup} norm) has an unconditional basis and it does not contain a copy of $ l_1. $ Therefore, as pointed out in \cite{ST}, there exists an equivalent renorming of this space with a strictly convex dual norm. Furthermore, we note that if a Banach space is reflexive under a particular norm, then it is reflexive under any equivalent norm. Combining these observations, we may and do conclude that there exists a norm on $ c_0 $ such that the space is nonreflexive and   the dual space is strictly convex.
\end{remark}

The next two theorems show how Birkhoff-James orthogonality($ T \bot_B A$) and strong Birkhoff-James orthogonality($T \bot_{SB}A$) are related in the space of bounded linear operators. A scalar $\lambda $ is said to belong to the approximate point spectrum of $A$, written as $\sigma_{app} (A),$ if there exists a sequence $\{x_n\} $ of unit vectors such that $ \|(A-\lambda I)x_n \| \longrightarrow 0.$ In the following theorem we show that if $0 \notin \sigma_{app} (A),$ then the notions $T \bot_B A$ and $T \bot_{SB} A$ are equivalent. 

\begin{theorem}\label{theorem:hilbert}
Let $\mathbb{H}$ be a Hilbert space, $T,A \in \mathbb{B}(\mathbb{H})$  and $0 \notin \sigma_{app} (A)$. Then $T \bot_B A$ and $T \bot_{SB} A$ are equivalent. 
\end{theorem}
\begin{proof}
Clearly $T \bot_{SB} A $ implies $T \bot_{B} A.$ On the other hand, if possible suppose that $T \bot_{B} A $ but $T\not \perp_{SB} A$. Then there exists a nonzero scalar $\lambda_0 \in \mathbb{R}$ such that $\|T\|= \|T- \lambda_0 A\|$. Now,  for all $\lambda \in \mathbb{R}$, $\|T- \lambda_0 A\| \leq \|T + \lambda A\|= \|(T- \lambda_0 A)+ (\lambda + \lambda_0) A\|$. Hence $(T- \lambda_0 A) \bot_B A$. Then from Paul \cite{P}, it follows that there exists $\{x_n\}\subseteq S_{\mathbb{H}}$ such that $\|(T-\lambda_0 A)x_n\| \rightarrow \|T- \lambda_0A\|$ and $\langle (T- \lambda_0 A)x_n, Ax_n \rangle \rightarrow 0$. Now, using the fact that $\|Ax_n\|\not\rightarrow 0$ (since $ 0 \notin \sigma_{app} (A)$) we have, $\|T- \lambda_0 A\|^2 = lim_{n \rightarrow \infty} \|(T- \lambda_0 A)x_n\|^2= lim_{n \rightarrow \infty} \{\|Tx_n\|^2 + |\lambda_0|^2 \|Ax_n\|^2 - 2 \lambda_0 \langle Tx_n, Ax_n \rangle \}= lim_{n \rightarrow \infty} \{\|Tx_n\|^2 - |\lambda_0|^2 \|Ax_n\|^2\} < lim_{n \rightarrow \infty} \|Tx_n\|^2 \leq \|T\|^2$. Therefore, $\|T- \lambda_0 A\|< \|T\|$, a contradiction. This proves the theorem.  
\end{proof}
 The result obtained in the previous theorem can be improved on to show that the notions $T \bot_B A$ and $T \bot_{SB} A$ are equivalent where $T,A \in \mathbb{B}(\mathbb{X}, \mathbb{Y}),$ if $0 \notin \sigma_{app} (A)$ and $\mathbb{Y}$ is a uniformly convex space. 
\begin{theorem}\label{theorem:uniform}
Let $\mathbb{X}$ be a normed linear space, $\mathbb{Y}$ be a uniformly convex space and $T,A \in \mathbb{B}(\mathbb{X}, \mathbb{Y})$.  If $0 \notin \sigma_{app} (A)$ then  $T\bot_B A $ and  $T\bot_{SB} A$ are equivalent.
\end{theorem}
\begin{proof}
Clearly $T \bot_{SB} A $ implies $T \bot_{B} A.$ On the other hand, if possible suppose that $T \bot_{B} A $ but $T\not \perp_{SB} A$.  Without loss of generality, assume that $\|T\|=1$.  Then from Theorem \ref{theorem:bounded}, it follows that  there exists sequences $\{x_n\}$, $\{y_n\} \subseteq S_{\mathbb{X}}$ and $\{\epsilon_n\}, ~\{\delta_n\}\subseteq [0,1)$ such that $\|Tx_n\|\rightarrow \|T\|$, $\|Ty_n\| \rightarrow \|T\|$ and $Ax_n \in (Tx_n)^{+(\epsilon_n)}$, $Ay_n \in (Ty_n)^{-(\delta_n)}$ and $\epsilon_n \rightarrow 0$, $\delta_n \rightarrow 0$.\\
 Then there exists a non-zero scalar $\lambda_0$ such that $\|T+\lambda_0 A\|=\|T\|$. Without loss of generality, assume that $\lambda_0 >0$. Then for any $\lambda \in (0,\lambda_0)$, $T+\lambda A= (1-\frac{\lambda}{\lambda_0}) T+ \frac{\lambda}{\lambda_0}(T+\lambda_0 A)$. Thus $\|T+\lambda A\|\leq (1-\frac{\lambda}{\lambda_0}) \|T\|+ \frac{\lambda}{\lambda_0}\|(T+\lambda_0 A)\|=(1-\frac{\lambda}{\lambda_0}) \|T\|+ \frac{\lambda}{\lambda_0}\|T\|=\|T\|\leq \|T+\lambda A\|$ since $T\bot_B A$. Therefore, $\|T+\lambda A\|=\|T\|$ for all $0\leq \lambda \leq \lambda_0 $. Thus for all $\lambda \in [0, \lambda_0],$ we have $\|T\|= \|T+\lambda A\|\geq \|(T + \lambda A) x_n\| \geq \sqrt{1-{\epsilon_n}^2}\|Tx_n\| \rightarrow \|T\|$. Hence  
 \begin{eqnarray}
 \|(T+ \lambda A)x_n\| \rightarrow \|T\|=1 ~ \forall ~\lambda \in [0, \lambda_0] 
 \end{eqnarray}
Now, suppose that $\inf \|Ax_n\|= c > 0,$ since $0 \notin \sigma_{app}(A).$ Take $\lambda_1, \lambda_2 \in [0,\lambda_0]$ such that $\lambda_1 \neq \lambda_2 $ and $|\lambda_1-\lambda_2|\leq \frac{2}{c}$. Now, choose $0< \epsilon < c|\lambda_1- \lambda_2|$. Then $\|(Tx_n + \lambda_1 Ax_n)- (Tx_n +\lambda_2 Ax_n)\| = |\lambda_1 - \lambda_2| \|Ax_n\| \geq |\lambda_1 - \lambda_2|c > \epsilon$. Again $\|Tx_n + \lambda_i Ax_n\|\leq \|T\|=1 \Rightarrow Tx_n +\lambda_i Ax_n \in B(\mathbb{Y})$ for $i=1,2$. Since $\mathbb{Y}$ is uniformly convex, clearly $\|\frac{Tx_n +\lambda_1 Ax_n}{2} + \frac{Tx_n + \lambda_2 Ax_n}{2}\|\leq 1- \delta_{\mathbb{Y}}(\epsilon),$ where $\delta_{\mathbb{Y}}(\epsilon) > 0 $ is the modulus of convexity of $\mathbb{Y}$. This implies that $\|Tx_n + \frac{\lambda_1+ \lambda_2}{2} Ax_n\|\leq 1-\delta_{\mathbb{Y}}(\epsilon)$. Thus $lim_{n \rightarrow \infty} \|Tx_n + \frac{\lambda_1+ \lambda_2}{2} Ax_n\| \leq 1-\delta_{\mathbb{Y}}(\epsilon) < 1,$ which contradicts (1). This completes the proof. 
\end{proof}
\begin{remark}
From Theorem \ref{theorem:rotund} and Example \ref{example:not rotund}, it follows that if $0 \in \sigma_{app}(A)$ then $ T \bot_B A$ may not imply $ T \bot_{SB}A.$
\end{remark}

\section{Smoothness of bounded linear operator}
To characterize smoothness of a bounded linear operator in the space of bounded linear operators is one of the most complicated and sought after area of study in the geometry of Banach spaces. Mathematicians including  Holub \cite{HO}, Heinrich \cite{HR}, Hennefeld \cite{HF}, Abatzoglou \cite{A}, Kittaneh and Younis \cite{KY}, Rao \cite{R}  have studied smoothness over the years, yet a  complete characterization of smoothness remain elusive. In \cite{PSG} a sufficient condition for the smoothness has been obtained for the first time. Extending the work done in \cite{PSG}, we here obtain separate  necessary and sufficient conditions for smoothness of a bounded linear operator.  We begin with the following theorem.

\begin{theorem}\label{theorem:ortho}
Let $\mathbb{X}, \mathbb{Y}$ be  normed linear spaces. Let $T \in \mathbb{B} (\mathbb{X}, \mathbb{Y})$ and $M_T = \{\pm x_0\}$ for some $x_0 \in S_ \mathbb{X}$ and $T$ further satisfies the following property:\\
Given any $\delta>0,$ if $\{H_\alpha : \alpha \in \Lambda\}$ is the collection of all hyperspaces such that $d( x_0, H_\alpha) > \delta$ then  $\sup\{\|Tx\| : x \in (\cup_{\alpha}H_\alpha) \cap S_\mathbb{X} \} < \|T\|$. Then for any $A \in \mathbb{B} (\mathbb{X}, \mathbb{Y})$, $T \bot_B A$ if and only if $Tx_0 \bot_B Ax_0$.
\end{theorem}
\begin{proof}
If $Tx_0 \bot_B Ax_0$ then clearly $T \bot_B A$, since $x_0 \in M_T$. For the other part, suppose that $T \bot_B A$ but $Tx_0 \not\perp_B Ax_0$. Then there exists a scalar $\lambda_0 \neq 0$ such that $\|Tx_0+ \lambda_0 Ax_0\|< \|T\|$. Then we can find an $\epsilon_1 >0$ such that \\
$$\|Tx_0+ \lambda_0 Ax_0\|< \|T\|- \epsilon_1.$$ 
Without loss of generality, we assume that $\lambda_0 > 0$. By continuity of the function $T+ \lambda_0 A$ at the point $x_0,$ we can find an open ball $B(x_0, r_0)$ such that \\
$$\|Tx+ \lambda_0 Ax\|< \|T\|- \epsilon_1 ~\forall x\in B(x_0, r_0).$$ \\
Let $\lambda \in (0, \lambda_0)$. Then for all $z\in B(x_0, r_0)\cap S_\mathbb{X}$,
    \[ Tz+ \lambda Az = (1-\frac{\lambda}{\lambda_0})Tz+ \frac{\lambda}{\lambda_0}(Tz+ \lambda_0 Az) \]
		\[\Rightarrow \|Tz+ \lambda Az\| \leq  (1-\frac{\lambda}{\lambda_0})\|Tz\|+ \frac{\lambda}{\lambda_0}\|(Tz+ \lambda_0 Az)\| \]
		\[\Rightarrow  \|Tz+ \lambda Az\| <  (1-\frac{\lambda}{\lambda_0})\|T\|+ \frac{\lambda}{\lambda_0}(\|T\|-\epsilon_1) =\|T\|-\frac{\lambda}{\lambda_0} \epsilon_1 \]
Since $\|(T+ \lambda A)z\|=\|(T+ \lambda A)(-z)\|$, we get for all $z\in (B(x_0, r_0) \cup B(-x_0, r_0))\cap S_\mathbb{X},$ 
\[ \|Tz+ \lambda Az\|< \|T\|- \frac{\lambda}{\lambda_0} \epsilon_1, \forall \lambda \in (0,\lambda_0). \]  
Consider $C=S_ \mathbb{X} \setminus( B(x_0, r_0) \cup B(-x_0, r_0))$. For any $z \in C$, $x_0$ and $z$ being linearly independent, from Theorem 2.3 of \cite{J}, there exists $\beta_z \in \mathbb{R}$ such that $x_0+ \beta _z z \bot_B z$. Therefore $\|x_0+ \lambda z\| \geq \|x_0+ \beta_z z\|= c_z $(say), for all $\lambda \in \mathbb{R}$. Clearly $c_z > 0$ as $\{x_0, z\}  $ is linearly independent. Define a function $f_z : $ Span $\{x_0, z\} \rightarrow \mathbb{R}$ by $f_z(a x_0 + b z)= a$. Clearly $f_z$ is a bounded linear functional on Span$\{x_0, z\}$ and so there exists a norm preserving extension  say $f$ of $f_z$ on $\mathbb{X}$. Clearly $\|f\|= \|f_z\| \leq \frac{1}{c_z}$. Let $H_z= ker (f)$. Then for all $y \in H_z$, $1=|f(x_0 -y)|\leq \|f\|\|  x_0 -y\|\leq \frac{1}{c_z}\| x_0 -y\| $. So $\| x_0-y\|\geq c_z$. Therefore $d( x_0, H_z)\geq c_z$. Thus for each $z \in C,$ there exists a hyperspace  $H_z$  of $\mathbb{X}$ containing $z$ such that $d(x_0, H_z)\geq c_z.$ So $C \subset (\cup_{z\in C}H_z) \cap S_\mathbb{X}.$ It is easy to check that there exists a $\delta>0$ such that $d(x_0, H_z)>\delta$ for all $z\in C$. \\

Now, by the hypothesis $\sup\{\|Tx\|: x \in (\cup_{z\in C}H_z) \cap S_\mathbb{X}\}< \|T\|$ and so $\sup\{\|Tx\|: x\in C\}< \|T\|$. Hence there exists $\epsilon_2 >0$ such that $\sup\{\|Tx\|: x\in C \}< \|T\|-\epsilon_2$. \\
Choose $ 0 < \widetilde{\lambda} < \min\{\lambda_0, \frac{\epsilon_2}{2\|A\|}\}$. Then for all $ z \in C ,$ we get 
\begin{eqnarray*}
\|Tz+\widetilde{\lambda} Az\| &\leq & \|Tz\|+|\widetilde{\lambda}|\|Az\| \\
                              & <  & \|T\|-\epsilon_2 +|\widetilde{\lambda}|\|A\|\\
															& < & \|T\| -\frac{1}{2}\epsilon_2
\end{eqnarray*}

Choose $\epsilon = \min \{\frac{1}{2}\epsilon_2, \frac{\widetilde{\lambda}}{\lambda_0} \epsilon_1\}$. Then for all $x \in S_{\mathbb{X}} $ we get, $$\|Tx+\widetilde{\lambda} Ax\| < \|T\|-\epsilon.$$
This shows that $ \|T+\widetilde{\lambda} A\| < \|T\|$, which contradicts the fact that $ T \perp_B A.$ 
This completes the proof. 
\end{proof}
We next prove a sufficient condition for the smoothness of a bounded linear operator.
\begin{theorem}\label{theorem:smooth suff}
Let $\mathbb{X}, \mathbb{Y}$ be  normed linear spaces. Let $T \in \mathbb{B} (\mathbb{X}, \mathbb{Y})$. Suppose the following conditions hold: \\
(i) $ M_T= \{\pm x_0\},$ for some $x_0\in S_\mathbb{X}$.\\
(ii) $ Tx_0$ is a smooth point of $\mathbb{Y}$. \\
(iii) Given any $\delta>0,$ if  $\{H_\alpha: \alpha \in \Lambda\}$ is the collection of hyperspaces such that $d( x_0, H_\alpha)> \delta$ then $ \sup\{\|Tx\| : x \in (\cup_{\alpha} H_\alpha) \cap  S_\mathbb{X} \} < \|T\|$.\\
Then $T$ is smooth.
\end{theorem}

\begin{proof}
Let $T\bot_B A_i(i=1,~2)$ with $A_i \in \mathbb{B}(\mathbb{X},\mathbb{Y})$. Then by Theorem \ref{theorem:ortho}, $Tx_0 \bot_B A_i x_0(i=1,~2)$. As $Tx_0$ is a smooth point of $\mathbb{Y}$ so $Tx_0 \bot_B (A_1+A_2)x_0$. Therefore, $T\bot_B (A_1+A_2)$. Then from Theorem 5.1 of James \cite{J}, it follows that  $T$ is smooth.
\end{proof}
On the other hand,  we prove that the following conditions are necessary for smoothness of a bounded linear operator $T.$

\begin{theorem}\label{theorem:smooth nec}
Let $\mathbb{X}, \mathbb{Y}$ be  normed linear spaces. Let $T \in \mathbb{B} (\mathbb{X}, \mathbb{Y})$ be a smooth point in the space of bounded linear  operators such that $M_T \neq \emptyset.$ Then \\
(i) $ M_T= \{\pm x_0\},$ for some $x_0\in S_\mathbb{X}$.\\
(ii) $ \sup\{\|Tx\| : x \in  H_\alpha \cap  S_\mathbb{X} \} < \|T\|$ for all $\alpha \in \Lambda$, where $\{H_\alpha: \alpha \in \Lambda\}$ is the collection of all hyperspaces such that $d(x_0, H_\alpha)> 0$.

\end{theorem}
\begin{proof} The condition (i)  follows from  Theorem 4.5 of Paul et al.\cite{PSG}. We need only to show (ii). 
If possible, suppose that $\sup \{\|Tx\| : x \in H_\alpha \cap S_\mathbb{X} \} = \|T\|$ for some $\alpha \in \Lambda.$ Then there exists a sequence $\{x_n\}\in H_\alpha \cap S_\mathbb{X}$ such that $\|Tx_n\|\rightarrow \|T\|$. Then each $z\in \mathbb{X}$ can be written as\\
$$z= a x_0+ h,$$ where $a \in \mathbb{R}$ and $h \in H_\alpha$. \\
Assume that $d(x_0,H_\alpha)=r>0$. So $ \inf \{\|x_0+h\|: h\in H_\alpha\} =  r$. Therefore for any $a \neq 0$, $\|x_0+ \frac{1}{a}h\|\geq r \Rightarrow \|z\|=\|a x_0+ h\|\geq r|a| \Rightarrow |a|\leq \frac{\|z\|}{r}$. Define linear operators $A_1, A_2: \mathbb{X}\rightarrow \mathbb{Y}$ as follows:
\[A_1 (z)= a Tx_0,~A_2 z= Th\]
It is easy to verify that $A_1,~ A_2$ are bounded linear operators. Now, for all $\lambda \in \mathbb{R}$, $\|T+ \lambda A_1\|\geq \|(T+\lambda A_1)x_n\|=\|Tx_n\|\rightarrow\|T\|$. Thus $T\bot_B A_1$. Similarly, for all $\lambda \in \mathbb{R}$, $\|T+ \lambda A_2\|\geq \|(T+\lambda A_2)x_0\|= \|Tx_0+ \lambda A_2 x_0\|=\|Tx_0\| = \|T\|$. Hence $T \bot_B A_2$. But $T=A_1+A_2,$ which shows that $T \not \perp_B (A_1+A_2)$. This shows that $T$ is not right additive and so from Theorem 5.1 of James \cite{J}, it follows that $T$ is not smooth.
\end{proof}
We would like to conclude the present paper with the following remark.
\begin{remark}
Although the necessary and sufficient conditions for the smoothness of a bounded linear operator mentioned in Theorems \ref{theorem:smooth nec} and \ref{theorem:smooth suff} are different, they are strikingly similar. Also the characterization of the smoothness of a bounded linear operator is still open when the norm attaining set is empty.
\end{remark}

\end{document}